\documentclass[a4paper,12pt,notitlepage]{amsart}
\usepackage{amsfonts, amsmath, amscd, amsthm, graphicx, amssymb}
\usepackage{tikz}
\usepackage{multicol}
\usetikzlibrary{arrows,shapes,decorations.markings,calc}

\usepackage{stackengine}
\usepackage{indentfirst}
\usepackage{mathtools}
\usepackage{tikz-cd} 
\usepackage{url} 
\usepackage{mathrsfs}
\usepackage{comment}
\newtheorem{theorem}{Theorem}[section]
\newtheorem*{theorem*}{Theorem}

\newtheorem{proposition}[theorem]{Proposition}
\newtheorem{lemma}[theorem]{Lemma}
\newtheorem{example}[theorem]{Example}
\newtheorem{corollary}[theorem]{Corollary}
\newtheorem{remark}[theorem]{Remark}
\newtheorem{problem}[theorem]{Problem}
\theoremstyle{definition}
\newtheorem*{acknowledgement*}{Acknowledgements}
\newtheorem*{problem*}{Problem}
\theoremstyle{definition}
\newtheorem{definition}[theorem]{Definition}

\makeatletter
\newcommand*{\rom}[1]{\expandafter\@slowromancap\romannumeral #1@}
\makeatother

\newcommand\stackrqarrow[2]{%
    \mathrel{\stackunder[2pt]{\stackon[4pt]{$\rightsquigarrow$}{$\scriptscriptstyle#1$}}{%
            $\scriptscriptstyle#2$}}}

\title{ Persistent pairs and connectedness in discrete Morse functions on simplicial complex \rom{1}}
\author{Chong Zheng}
\address{Faculty of  Science and Engineering, Waseda University, 
Ohkubo, Shinkuju-ku, Tokyo, 169-8555 Japan}
\email{c{\_}zheng@aoni.waseda.jp}

\begin{document}

\begin{abstract}

In this paper,  we study some useful properties of persistent pairs in
a discrete Morse function on a simplicial complex $K$. 
In case of $\dim K=1$ (i.e., a graph), 
by using the properties, 
we characterize 
strongly connectedness of critical simplices 
between two distinct 
discrete Morse functions, 
and relate the number of such pairs to the Euler characteristic of $K$. 

\end{abstract}
\maketitle

\section{Introduction}
Introduced by R. Forman \cite{Forman_guide}, 
discrete Morse theory is  a combinatorial 
version of standard Morse theory.
The main idea of discrete Morse theory is 
to collapse pairs of simplices in such 
a way that the homotopy type of 
the space remains consistent, 
as
similar to the elementary collapse theory invented
by J.H.C. Whitehead \cite{jhc}.
 A \textit{discrete Morse function} 
 is  a real-valued function 
 assigned to the set of simplices.
Also, a discrete Morse function corresponds
 to a \textit{gradient vector field},
in which non-critical simplices are paired with each other.
We can reduce the number of simplices in a complex
by collapsing the pairs in a gradient vector field.
Hence, discrete Morse theory is widely used for
the computation of the homology of spaces.
The main theorems of discrete Morse theory 
are first formulated for 
CW complexes,
and have since been expanded to
other structures,
including simplicial complexes \cite{Discrete Morse Theory on Simplicial Complexes}, 
graphs \cite{Discrete Morse theory on graphs},
and posets \cite{poset}.
This approach has proven to be useful 
in studying the topology of spaces and has found numerous 
applications in various areas, notably in the calculation of
persistent homology \cite{homology}.

King, Knudson, and Mramor Kosta \cite{birth and death} introduced
 the concept of 
 \textit{birth and death theory} in
  discrete Morse theory,
which has been applied to 
several interesting data analysis problems.
 In their framework,
given a CW complex and a series of discrete Morse functions,
they define a \textit{connectedness} relationship between two critical simplices 
in different gradient vector fields.
A series of organized discrete Morse functions on the 
same cellulation
divides space into slices in which critical 
simplices appear and disappear.
As time varies,  they trace  critical points via a 
discrete version of a bifurcation diagram.
The diagram provides information 
on the birth and death of each critical simplex.
The persistent time of each critical simplex
reflects the geometric information of
the  space under the given discrete Morse 
functions.

In this paper, 
we apply persistence pairs 
to discuss various problems related to  \textit{connectedness} and
\textit{strong connectedness}, including one posed in 
\cite{birth and death}.
A
\textit{persistence pair}, 
is a pair of simplices used 
to depict the birth and death of a homology class with respect to 
ascending indices of the filtration of the simplicial complex.
The concept of persistence pair can be found 
in \cite{Discrete Morse Theory scoville, bauer}.
Furthermore,  a similar concept called \textit{homological sequence} 
is
discussed in 
the case of  graphs 
in Section 3 and 5 of  Ayala et al \cite{Discrete Morse theory on graphs}.
Compared with \cite{Discrete Morse theory on graphs} 
and \cite{chari},
in which the authors investigate the number of 
 discrete Morse functions on a graph,
our primary objective  in this paper
 is to use persistence pairs to 
study the  strong connectedness between 
 different discrete Morse functions on a graph.

Section 2 provides a 
brief summary of the definitions and results in discrete Morse theory. 
In Section 3,
we investigate persistence pairs
of discrete Morse functions,  and formalize several properties 
of persistence pairs in terms of simplicial complexes.
In Section 4,
we discuss the 
idea of connectedness between multiple discrete Morse 
functions and 
investigate several problems regarding to the connectedness, including
 uniqueness of strong connectedness and gradient paths.
Moreover,  in Theorem \ref{main theorem}, 
we study the relationship 
between the number of 
strong connectedness pairs and the Euler characteristic of a graph.  Precisely, 
 \begin{theorem*}

Let $G$ be a  graph
and 
$$f_1, f_2 : G \to \mathbb{R}$$
be discrete Morse functions.
Let $A_q^{f_1,f_2}(G)$ be the number of strongly connected pairs of 
$q$-dimensional simplices between $f_1$ and $f_2$,  $q=0,1$.
Then, 
$$A_0^{f_1,f_2}(G)- A_1^{f_1,f_2}(G) =\beta_0(G)- \beta_1(G)=\chi(G),$$
where $\chi(G)$ is the  Euler characteristic of the graph $G$.
 
 \end{theorem*}

\section{Prelimanaries}
This section is devoted to briefly summarize a 
theoretical basis of discrete Morse theory.
For details, the readers may refer to 
\cite{Morse Theory for Cell Complexes} and \cite{Forman_guide}.
Additional resources on discrete Morse theory 
can be found in textbooks such as \cite{Discrete Morse Theory scoville}
and 
\cite{Discrete Morse Theory Kozlov}.
This paper includes some basic definitions and
facts about homology theory,
which is 
a fundamental field in algebraic topology 
and has wide applications in
various fields.
Some of the details regarding homology theory are 
left undefined and unproven.
For details on basic homology theory, 
one refers to \cite{Munkres}.

A \textit{simplicial complex} $K$ is 
a collection of simplices such that
 every face of a simplex of $K$ is in $K$,
 and the intersection of any two simplices of $K$ is a face of each of them.
By $K_p$ we denote the set of simplices 	 
of dimension $p$.
The \textit{dimension} of $K$ is the
maximum dimension among its simplices.

We denote the chain complex  of $K$ with coefficient in a 
 finite field $\mathbb{F}$ 
 by $C_{\ast}(K;\mathbb{F}):= \bigoplus_q C_q(K; \mathbb{F})$
 with boundary operator $\partial_{\ast}$.
We denote by $B_q(K):=\text{im}(\partial_{q+1})$,
and  $Z_q(K):=\text{ker}(\partial_q)$. 
We then define the \textit{$q$-th homology group} of $K$ as
$H_q(K):= Z_{q}(K)/ B_q(K)$.
The rank of $H_q$ is called the $q$-th \textit{Betti number}. 
Each element $[h]\in H_q$ is called a \textit{homology class}.
A chain $c\in Z_q$ is called a \textit{representative chain} of $[h]$,
if $c$ is in the homology class $[h]$.
Two chains $h_1, h_2 \in C_q$ are said to be \textit{homologous} if 
they both are  representative chains of the same homology class.

Now we introduce the fundamental definitions and theorems 
of discrete Morse theory.
\begin{definition}

A \textit{discrete Morse function} on $K$ is a real-valued 
function 
$$f: K\longrightarrow \mathbb{R}$$
satisfying for all
$\sigma \in K_p$,
\begin{enumerate}
\item $\displaystyle \# \{\tau^{(p+1)} \succ \sigma | f(\tau) \leq f(\sigma)\} \leq 1;$
\item 
$\displaystyle  \# \{v^{(p-1)} \prec \sigma | f(\sigma) \leq f(v)\} \leq 1,$
\end{enumerate}
where $\#$ denotes the cardinality of a set.
Notations $v \prec \sigma$  and  $\tau  \succ  \sigma$ mean  
that \textit{$v $ is a face of $\sigma$} and \textit{$\sigma$ is a face of $\tau $},
respectively.
Also, we use the notation $\sigma^{(p)}$ to represent that the dimension 
of simplex $\sigma$ is $p$.
\end{definition}

\begin{definition}
Given a discrete Morse function $f$
on $K$,
we say that a simplex
$\sigma \in K_p$ is a 
\textit{critical} $p$-dimensional simplex (of $f$)
if
\begin{enumerate}
\item 
$\displaystyle\# \{\tau^{(p+1)} \succ \sigma \, | \, f(\tau) \leq f(\sigma)\} =0;$
\item 
$\displaystyle\# \{v^{(p-1)} \prec \sigma \, | \, f(\sigma) \leq f(v)\} =0.$
\end{enumerate}
If $\sigma$ is a critical simplex,
then we call $f(\sigma)$ a \textit{critical value} of $f$.
Conversely, if a simplex $\sigma$ is not critical, 
then
we say that $\sigma$ is a \textit{non-critical} simplex,
and $f(\sigma)$ is a \textit{non-critical value}.

A discrete Morse function is called 
\textit{excellent}, if 
its all critical values are different. 
In this paper, 
we always assume that the discrete Morse functions 
under consideration are excellent without loss of generality,
because we can always perturb the real values to 
ensure distinct real values.

\end{definition}

\begin{definition}
Let $f$ be a discrete Morse function.
A \textit{gradient vector field} $V$ corresponding to $f$ is 
the collection of pairs 
$(\alpha^{(p)}, \beta^{(p+1)})$
satisfying 
$\alpha^{(p)} \prec \beta^{(p+1)}$ and
$f(\alpha) \geq f(\beta)$.
We call such a pair $(\alpha^{(p)}, \beta^{(p+1)})$
a \textit{gradient pair}.
Given a gradient vector field $V$,
a  \textit{gradient $V$-path} is a sequence 
of simplices
$$\alpha_0^{(p)}, \beta_0^{(p+1)}, \alpha_1^{(p)}, \beta_1^{(p+1)},\cdots, \beta_r^{(p+1)}, \alpha_{r+1}^{(p)} $$
such that for
each $i= 0, \cdots,r $,
$(\alpha^{(p)}, \beta^{(p+1)})\in V$
and
$\beta_i^{(p+1)} \succ \alpha_{i+1}^{(p)} \neq \alpha_i^{(p)}.$

\end{definition}

By Definitions 2.2 and 2.3,
 a  simplex is in the gradient vector field if 
and only if it is 
non-critical. 
Consequently,  if a simplex is $f$-critical (non-critical), it can also be referred to as
 \textit{$V$-critical (non-critical)}.

We denote a \textit{gradient  $V$-path} $\gamma$ 
consisting of $p$ and $(p+1)$-dimensional 
simplices
from 
$\alpha_0^{(p)} $ to 
$\alpha_{r+1}^{(p)}$
by 
$$\gamma: \alpha_0^{(p)}  \stackrqarrow{V(p,p+1)}{}  \alpha_{r+1}^{(p)}. $$
Note that we can also take any $(p+1)$-dimensional simplex
$\beta_0^{(p+1)}$ as the starting simplex, and
$\beta_r^{(p+1)}$ as the ending simplex.
Thus, we can denote the corresponding 
path consisting of $p$- and $(p+1)$-dimensional simplices
by 
$$\gamma: \beta_0^{(p+1)} \stackrqarrow{V(p,p+1)}{}  \beta_r^{(p+1)}. $$

We say that two gradient paths $\gamma_1,\gamma_2$ are the \textit{same path} 
if all elements of two paths are the same, 
denoted $\gamma_1=\gamma_2$.
A \textit{trivial path} is a gradient $V$-path that consists of only one simplex. 

\begin{definition}
\label{subcomplex}
For $c\in \mathbb{R}$,
we define the \textit{sub-level complex }
$$ K(c)= \bigcup_{\sigma\in K , f(\sigma)\leq c } \bigcup_{\tau \leq \sigma} \tau.$$
$K(c)$ is the sub-complex of $K$ consisting of all simplices 
$\tau$ with 
$f(\tau)\leq c$ and their faces.
\end{definition}

The followings are fundamental theorems 	in discrete Morse theory.
Proofs can be found in \cite{Forman_guide, Morse Theory for Cell Complexes}.

\begin{theorem}[\protect{\cite[Thm. 3.3]{Morse Theory for Cell Complexes}}]

If $a<b$ are real numbers such that $[a, b]$ contains no
critical values of f, then
$K(b)$ is homotopy equivalent to $K(a)$.
\label{collapse theorem}

\end{theorem}



\begin{theorem}[\protect{\cite[Thm. 2.11]{Forman_guide}}]
\label{Morse_inequalities}
Let $C_q$ be the number of critical simplices of dimension $q$.
Then,
\begin{enumerate}
\label{c>b}
\item
for any 
$i\in \mathbb{N}$,
$$C_i -C_{i-1}+\cdots \pm C_0 \geq \beta_i - \beta_{i-1} + \cdots \pm \beta_0,$$
where $\beta_q$ is the $q$-th Betti number. 
\vspace{0.5cm}
\item  for any $i\in \mathbb{N}$,
$$C_i\geq \beta_i.$$

\item $\chi(K)= \beta_0 -\beta_1 +\beta_2 -\cdots \pm \beta_{dim K}
=C_0 -C_1+C_2-\cdots \pm C_{dim K},$
\vspace{0.5cm}
where $\chi(K)$ is the Euler characteristic number of $K$.
\end{enumerate}
\end{theorem}

The inequalities above are commonly referred to as the
 \textit{weak discrete Morse inequality} and  the \textit{strong discrete Morse inequality},
respectively.



\section{Persistence pairs of discrete Morse functions }

Given a discrete Morse function
$f: K\to \mathbb{R}$, we recall
the concept of a \textit{persistence pair}, which consists
 of two critical simplices of $f$
 and describes the birth and death of a homology class.
Then,  we study some useful properties of persistence pairs.
Some of them are not complicated when we consider them 
at the 
level of discrete Morse complexes, see Forman \cite{Forman_guide} 
 and Section 3.1 of Bauer \cite{ bauer}.
 Nonetheless, 
 in order to thoroughly examine the connectivity and strong connectivity
 between critical simplices of different discrete Morse functions,   
 which is our primary goal in Section 4, 
 a clear description at the level of simplicial complex $K$
 is necessary.

We first define the \textit{filtration} associated to $f$
 of a simplcial complex 
and its homology in discrete Morse theory
by
 applying Definition \ref{subcomplex}.
Recall we always assume discrete Morse functions 
to be excellent.

 \begin{definition}
 
Given a discrete Morse function
$f: K\to \mathbb{R}$,
let $c_1< c_2 < \cdots < c_n$
be a sequence of increasing real numbers
and  $K(c_i)$ be the sub-level complex at value $c_i$.
We call 
$$ K(c_1) \subset K(c_2) \subset \cdots \subset K(c_n)$$
  a \textit{filtration} of  $K$ associated to $f$,
  where $K(c_n)=K$.
For convenience, 
if $f(\sigma)= c$, 
 then we  use both $K(c)$ and $K(\sigma)$ to 
represent the sub-level complex at value $c$ or 
$f(\sigma)$.

 \end{definition}

Note that if a discrete Morse 
function 
$f: K \to \mathbb{R}$ is restricted to a sub-complex $K_i$, then 
$f|_{K_i}$ is still a discrete Morse function.
A discrete Morse function provides local information
of  $K$.
Specifically,
whether a simplex of $K$ is critical or not 
remains consistent when we  consider each sub-complex $K(c_i)$ of $K$.
This property is formalized in the following lemma.

\begin{lemma}
Let 
$ K(c_1) \subset K(c_2) \subset \cdots \subset K(c_n)$
be a filtration of  $K$ associated to $f$.
If $\sigma$ is a critical (or non-critical) simplex of $K(c_i)$ for some $i$,
then $\sigma$ is also a critical (or non-critical) simplex of $K(c_j)$
 for all $j>i. $

\end{lemma}

\begin{proof}
To prove the critical case, 
we let $\sigma$
be a $q$-dimensional critical simplex in $K(c_i)$
and  $j>i$.
Since $\sigma$
is a critical simplex of $K(c_i)$,
by definition,
\begin{enumerate}
\item  for any $\alpha^{(q-1)}\in K(c_i)$ with $\alpha^{(q-1)}\prec \sigma$,
$f(\alpha)<f(\sigma)$;

\item  
for any $\tau^{(q+1)} \in K(c_i)$ with $\tau^{(q+1)}\succ \sigma$,
$f(\tau) > f(\sigma).$
\end{enumerate}
The second statement still holds for all $\tau^{(q+1)}\in K(c_j)$,
since $j>i$.
Furthermore,
by the definition of sub-level complex,
since $\sigma$ is  in $K(c_i)$,
all the faces of $\sigma$ are also in $K(c_i)$.
Hence, 
statement (1) still holds for $K(c_j)$.

A similar proof can also be applied to the non-critical case.

\end{proof}
Note that the statement in Lemma 3.2 may fail to hold for 
general objects other than simplicial complexes. 
 In paper \cite{poset},  the authors define the 
 discrete Morse theory for posets, 
 and a similar statement  to Lemma 3.2 for posets 
 does not hold.

By applying the simplicial homology functor to the filtration, 
we obtain a sequence of homology groups and inclusion-induced
 homomorphisms. 
For each homology class, 
we can use a pair of simplices to describe its birth and death.
A homology class $[h]$ is said to \textit{be born} at 	
a certain $f$-value $c_i$,
if 
$[h]$ is not in the image of $H_{\star} (K(c_{i-1}))$. 
On the other hand, 
a homology class $[h]$ 
is said to \textit{persist} at  $c_i$, 
if the image
of $[h]$ in $H_{\star} (K(c_{i+1}))$ is non-zero, 
otherwise, 
it is said to \textit{die} at  $c_i$.
If there exists some simplex $\sigma$ with
$f(\sigma)=c_i$, 
then we also say that 
the homology class $[h]$
is born at $\sigma$, or dies at $\sigma$.
To precisely describe the persistence of a cycle,
\textit{the elder rule} is applied.
The elder rule is to let the 
death simplex always be paired
 with the youngest birth simplex. 

\begin{definition}
A pair of simplices
$(\sigma, \tau)$ is called a \textit{persistence pair} if 
there is a homology class 
$[h]$
that is born at $\sigma$ and dies at $\tau$.
 $\sigma$ is referred to as the \textit{birth simplex}, 
 and $\tau$ is referred to as the \textit{death simplex} of $[h]$.
In the case where $[h]$ is born at $\sigma$ and does not die,
 we say that $\sigma$ is \textit{paired with infinity},
denoted $(\sigma, \infty)$.

\end{definition}

This  is a well-defined definition because we 
assume the discrete Morse function being
excellent and the elder rule is applied.
Note that in \cite{bauer},
the birth simplex and death simplex are also 
called the \textit{positive
cell} and \textit{negative cell}, respectively.

We call the dimension of a homology class
 $[h]$ to be  the 
dimension of  the simplices of a representative chain.
Then, the following statements about persistent pairs 
hold true under the conditions of discrete Morse 
function and the elder rule.
\begin{proposition}
\label{1st_cell}
Let $f$ be a discrete Morse function on $K$.
Suppose that  $[h]$ is a $q$-dimensional homology class with persistent pair 
$(\sigma, \tau)$.
Then 
\begin{enumerate}
\item $ \dim \sigma =q$;

\item $\dim \tau =q+1$;

\item $[h]$ is represented by a cycle 
$z=\sum_i n_i  \sigma_i $
such that
$\sigma= \sigma_j$,
for some $j$.
\end{enumerate}

\end{proposition}

\begin{proof}
To prove the first statement,
since the $\dim \sigma$
can not be smaller than $q$,
we assume the opposite,  that is,  $\dim \sigma > q.$

Since $[h]$ is born at 
$\sigma$,
there exists at least one $\alpha\prec \sigma$,
such that $\dim \alpha =q$, and $\alpha$ is a part of some representative chain of 
$[h]$.
We write the chain $z$ as 
$z= \sum_i n_i \alpha_i,$
where there exists at least one index $j$ such that $\alpha=\alpha_j$
and $n_j=1.$

If there are multiple such $\alpha$ satisfying the conditions, 
without loss of generality,  we label them as
$\alpha_1, \cdots, \alpha_K$.
Considering any two $\alpha_k\prec \sigma$ and $\alpha_l\prec \sigma$ of them,
since 
$\sigma$ is a simplex, 
by definition,
there exists a $(q+1)$-dimensional simplex $\delta$ with  $a_k$ and $a_l$ 
as faces. 
Since $[h]$ is born at 
$\sigma$,
we have
$f(\sigma)< f(\delta^{(q+1)}) < f(a_k^{(q)}),$
and 
$f(\sigma)< f(\delta^{(q+1)}) < f(a_j^{(q)}),$
which contradicts
the definition of a discrete Morse function.

On the other hand, if there is only one such $\alpha$,
since $\sigma$ is a simplex, 
there exists a $(q+1)$-dimensional face $\delta^{(q+1)}\prec \sigma$,
such that $\alpha \prec \delta$ and $f(\alpha) > f(\delta)$.
Let $z= \sum_i n_i \alpha_i$ be a representative chain of $[h]$ with
 $\alpha= \alpha_j$ and $n_j=1$ for some index $j$. 
Since $f$ is a discrete Morse function,  by definition,
$f(\alpha) > f(\delta) > f(\tilde{\alpha}) $,
for any $\tilde{\alpha} \neq \alpha$ in the boundary of $\delta$. 
Let $\tilde{z}$ be the chain consisting of all elements  in $z$ 
except $\alpha$,
and all elements of the boundary of $ \delta$ except $\alpha$.
 Then $z$ is homologous to $\tilde{z}$,
contradicting
 the fact that $[h]$ is born at $\sigma$
as $\tilde{z}$ appears earlier than $\sigma$.

A similar proof also applies to statement (2).
The third statement is straightforward after proving (1) and (2).
\end{proof}

\begin{proposition}
\label{codimension 1}
Let $f$ be a discrete Morse function.
If $(\sigma^{(q)}, \tau^{(q+1)} )$ is a persistence pair  of $f$,
then both $\sigma$ and  $\tau$ are $f$-critical.

\end{proposition}

 \begin{proof}
 Suppose that $(\sigma^{(q)}, \tau^{(q+1)} )$ is a persistence pair.
To prove $\sigma$ is $f$-critical,
consider any $(q-1)$-dimensional face
  $\alpha^{(q-1)} \prec \sigma$.
By Proposition \ref{1st_cell} (3),
there exists a cycle 
$z= \sum_i n_i \sigma_i,$
with $\sigma_j=\sigma, n_i=1$, for some index $j$.
Since $\sigma$ is a simplex,
any $(q-1)$-dimensional face
  $\alpha^{(q-1)} \prec \sigma$ is the intersection of  
$\sigma$ and $\sigma_k$ for some $k$ in the cycle $z$.
Note that $f(\sigma) > f(\sigma_k),$
as $z$ is born at $\sigma$.
Thus, by the definition of discrete Morse function, 
$f(\sigma)>f(\alpha).$
The proof for the 
statement that 
for any $(q+1)$-dimensional simplex $\delta^{(q+1)}$ 
such that $\delta^{(q+1)} \succ \sigma$, 
it has a larger $f$-value than $\sigma$, 
is similar to the proof given in Proposition \ref{1st_cell}.

Next, we prove that $\tau$ is $f$-critical.
By Proposition \ref{1st_cell} (2),
we have
$\dim \tau =q+1.$
Suppose $ \rho^{(q)} \prec \tau$.
Since $(\sigma, \tau )$ is a persistence pair,
every element $\rho_i $ of the boundary of $ (\tau)$ satisfies
$f(\tau)> f(\rho_i).$
The rest is similar to the proof given in Proposition \ref{1st_cell}.
 \end{proof}
 
According to  Proposition \ref{codimension 1},  we can classify critical simplices
based on the structure of their persistence pairs.
We define the \textit{dimension} of a persistence pair as the 
dimension of its birth simplex.

\begin{theorem}
Let $\sigma^{(q)}$ be a critical simplex.
Then, $\sigma^{(q)}$ is in exactly 
one  of the three following kinds of
persistence pairs:
\begin{enumerate}
\item  a $q$-dimensional persistence pair 
$(\sigma^{(q)}, \tau^{(q+1)})$;
\item  a $(q-1)$-dimensional persistence pair $(\alpha^{(q-1)},\sigma^{(q)})$;

\item  a $q$-dimensional persistence pair consisting of  infinity 
$(\sigma^{(q)}, \infty)$. 
\end{enumerate}
\end{theorem}

Note that in \cite{Discrete Morse theory on graphs}, when $K$ is a graph, 
the type (3) critical simplices are referred to as
 \textit{essential vertices} and \textit{essential edges} 
 when $q = 0$ and $q = 1$, respectively. 
The other types of critical simplices are called \textit{superfluous}. 
 
A gradient pair  provides information locally
because it only captures of the ordering
 information of the neighbourhood
of an element. 
A persistent pair, on the other hand, contains
 global order information for a specific dimension,
hence it provides global information.
Notably,
only  critical simplices paired with 
infinity contribute to the homology  of  the underlying space,
as those cycles persist until the end of the filtration.
Recall  Theorem \ref{Morse_inequalities}, the Morse inequalities, 
which establishes
a connection between the number of $q$-dimensional
critical simplices and  $q$-th Betti number.
Therefore,  as a corollary, 
we can use the structures of persistence pairs to 
re-write the Morse inequalities as equations,
providing a more refined understanding of the relationship 
between critical simplices and homology.
A similar result for the graph case is given in 
\cite{Discrete Morse theory on graphs}.

Let $P_q$ be the set of persistence pairs whose 
birth simplices have dimension $q$.
 Let $\hat{P}_q$  be the set of persistence pairs whose 
birth simplices have dimension $q$
and death simplices are not infinity.
Then, $\hat{P}_q \subset P_q,$
and the $q$-th Betti number is given by
$\beta_q= \# P_q- \#\hat{P}_q. $

\begin{corollary}
\label{new morse inequality 1}
The  Morse inequality in Theorem \ref{c>b} can be re-written as an equation
$$C_q = \beta_q + \#\hat{P}_{q-1} + \#\hat{P}_q.$$

\end{corollary}

Moreover, the strong and  weak Morse inequalities
become more explicit as well.

\begin{corollary}
\label{new morse inequality 2}
Let $K$ be a simplicial complex, and the dimension of $K$ is $n$. 
Then 
\begin{equation*}
C_n- C_{n-1}+\cdots \pm C_0= \beta_n - \beta_{n-1} +\cdots \pm \beta_0.
\end{equation*} 

Also, 
for $i= 0,1,\cdots, n,$

\begin{equation*}
C_i - C_{i-1} + \cdots \pm C_0  = \beta_i -\beta_{i-1} + \cdots \pm \beta_0 + \# \hat{P}_i .
\end{equation*}

\end{corollary}

\begin{proof}
We first prove the second equation.
For any $i= 0,1,\cdots, n,$
we calculate the left side of 
the equation using Corollary \ref{new morse inequality 1}.

\begin{align*} 
C_i- C_{i-1} + \cdots \pm C_0 &=(\beta_i + \#\hat{P}_{i-1} + \#\hat{P}_i )  \\
&- (\beta_{i-1} + \#\hat{P}_{i-2} + \#\hat{P}_{i-1} )  \\
&+\quad \quad \quad\vdots     \\
&\pm (\beta_0 + \#\hat{P}_{-1} + \#\hat{P}_0 ). \\
\end{align*}
Since 
$\#\hat{P}_{-1}  =0, $
the second equation holds.

For the first equation, suppose that $i=k$.
Then, since  $\#\hat{P}_n=0$,
the first equation holds, as well.
\end{proof}


\section{Strong connection of discrete Morse functions on graphs}
In this section, we explore the relationship between multiple 
discrete Morse functions on a simplicial complex.
We begin this section by defining  
 connectedness between critical simplices 
  in different 
 gradient vector fields.
 The definition is initially
introduced in \cite{birth and death}. 
The connectedness relationship is a binary relationship (but not an
equivalence relationship)
 between two critical simplices with the same dimension
 in two distinct gradient vector fields.
Consequently,  the connection relationship between two simplices also 
reveals  the relationship between the corresponding
functions.
The precise definition is given as follows.

\begin{definition}{\cite{birth and death}.}
\label{def_connection}
Let $K$ be a simplicial complex,  and
$\{f_i\}$ be a family of discrete Morse functions on
$K$.
For each $i$,  let $V_i$ be the corresponding 
gradient vector field of $f_i$.
Suppose that
$\alpha$ and $\beta$
are $q$-dimensional simplices of $K$, and 
$\alpha$ is critical in $V_i$, 
and $\beta$ is critical in $V_j$, $j\neq i$.

We say that $\alpha$ is \textit{connected to} $\beta$ 
if there are a $q$-dimensional simplex $\sigma$
and gradient paths
$\alpha \stackrqarrow{V_i(q-1,q)}{} \sigma$
and 
$\sigma \stackrqarrow{V_j(q,q+1)}{} \beta.$
We denote by 
$\alpha \to \beta $,
if 
$\alpha$ is connected to $\beta$.
 We also define that 
$\alpha$ is \textit{strongly connected} to $\beta$ 
if $\alpha$ is connected to $\beta$
and $\beta$ is connected to $\alpha$,
denoted 
$\alpha \leftrightarrow \beta $.
\end{definition}

For more detailed information and applications of connectedness, 
one  refers to \cite{birth and death}.
Note that by definition, 
$\alpha \leftrightarrow \alpha$.
Additionally, if $\alpha$ is critical at both $V_i$ and $V_j$,
then $\alpha$ is strongly connected to itself 
by the trivial path.

Let $f_1,f_2$ be discrete Morse functions on $K$,
$\tilde{\sigma}_1$, $\tilde{\sigma}_2$  be $f_1$,$f_2$-critical simplex,
respectively. 
Given the concept of connectedness between critical 
simplices in two different gradient vector fields,
it is natural to consider the following problems:
\begin{enumerate}
\item Let $A_q^{f_1,f_2}(K)$ be the number of strongly connected pairs of 
$q$-dimensional simplices between $f_1$ and $f_2$.
Does $A_q^{f_1,f_2}(K)$ depend on $\beta_q(K)$?

\item Suppose that $\tilde{\sigma}_1$ is strongly connected to $\tilde{\sigma}_2$.
Under what conditions is the gradient path connecting them unique?

\item 
(Posed in \cite{birth and death}) 
Suppose that $\tilde{\sigma}_1$ is strongly connected to $\tilde{\sigma}_2$.
Under what conditions is $\tilde{\sigma}_1$  
 not strongly connected to any other simplices of $f_2$ except  $\tilde{\sigma}_2$?
\end{enumerate}

In this section,  we will  investigate the connectedness relationship in 
the case of simple graph and provide  solutions to the aforementioned problems. 
A  \textit{simple graph} $G$ is a graph  that 
does not contain loops or parallel edges, 
making it equivalent to a one-dimensional simplicial complex.
Discrete Morse theory on 
graphs
is not complicated, and 
several interesting and fundamental properties 
 are introduced in  Kozlov \cite{Discrete Morse Theory Kozlov} 
 and  Ayala et al \cite{Discrete Morse theory on graphs}.
 
 With regard to the connectedness 
 between $\tilde{\sigma}_1^{(q)}$ and $\tilde{\sigma}_2^{(q)}$
 on a graph,
 when $q=0$,
 $\tilde{\sigma}_1$ is  connected to $\tilde{\sigma}_2$
 if there is a gradient path
 $\tilde{\sigma}_1 \stackrqarrow{V_2(0,1)}{} \tilde{\sigma}_2$;
 when $q=1$,
 $\tilde{\sigma}_1$ is  connected to $\tilde{\sigma}_2$
 if there is a gradient path
 $\tilde{\sigma}_1 \stackrqarrow{V_1(0,1)}{} \tilde{\sigma}_2$;

\begin{remark}
To address the higher dimensional cases,
we plan to present another paper \cite{no.2}
to give solutions to the aforementioned problems.

\end{remark}

We begin with a simple case by assuming that
the discrete Morse functions
$f_1,f_2$ are  optimal. 
We define a discrete Morse function to be \textit{optimal (or perfect)} 
if, for each $n$, 
the number of $n$-dimensional critical simplices $C_n$
equals the 
$n$-th Betti number $\beta_n$
 \cite{Discrete Morse theory on graphs}.
 For $i=1,2$, we use $\tilde{v}_i,\tilde{e}_i$
to represent $0$-, $1$-dimensional $f_i$-critical simplices,
respectively.

\begin{proposition}
\label{optimal}

Let $G$ be a graph
and 
$$f_1, f_2 : G \to \mathbb{R}$$
be optimal discrete Morse functions,  $q=0,1$.
For each  $q$-dimensional homology class $ [  \alpha ]$,
suppose that $ [  \alpha ]$ is born at $\tilde{\sigma}^{(q)}_1$ in $f_1$  and  
$\tilde{\sigma}^{(q)}_2$ in  $f_2$.
Then, $\tilde{\sigma}_1$ is strongly connected to $\tilde{\sigma}_2.$
\end{proposition}

\begin{proof}
We assume that $\tilde{\sigma}_1\neq\tilde{\sigma}_2$, as the result is 
trivial when $\tilde{\sigma}_1 = \tilde{\sigma}_2$.
Since $[\alpha]$ is born at $\tilde{\sigma}_1$ in $f_1$
and  at $\tilde{\sigma}_2$ in $f_2$,
 $\tilde{\sigma}_1$ and $\tilde{\sigma}_2$ are in persistent pairs 
 $(\tilde{\sigma}_1, \infty) $ and   $(\tilde{\sigma}_2, \infty) $ of $f_1$ and $f_2$, respectively.

When $q=0,$ to show $\tilde{v}_1 \to \tilde{v}_2,$ it suffices to
  show the existence of an $f_2$-gradient path connecting $\tilde{v}_1$ and $\tilde{v}_2$.
Since $f_2$ is optimal, there are no other
$(v,\infty)$ or $(v,e)$ type critical simplices in the connected component 
generated by $ [  \alpha ]$.
Therefore, it suffices to show that there is a path from $\tilde{v}_1$ to $\tilde{v}_2$ in which 
all $1$-dimensional simplices $e$ are non-critical.

Suppose that there is an $f_2$-critical simplex $e$ in the path.
Then,
 $(e,\infty)$ is an $f_2$-persistence pair, and we can always find an
alternative 
path that consists of the elements of the cycle and flows towards to $\tilde{v}_2$.
As $G$ is finite, we can repeat this procedure to obtain the desired gradient path. 

When $q=1,$ 
$\tilde{e}_1$ and $\tilde{e}_2$ belong to   the same cycle. 
Since $[\alpha]$ is born at $\tilde{e}_1$ in $f_1$,
$f_1(\tilde{e}_1)>f_1(\tilde{e}_2)$.
When $f_1$ is optimal, 
there is always  one  gradient $f_1$-gradient path connecting 
$\tilde{e}_1$ and $\tilde{e}_2$. 
Thus, $\tilde{e}_1 \to \tilde{e}_2$.

Therefore, since both $ f_1$ and $ f_2 $ are optimal, 
we have
$\tilde{v}_1 \leftrightarrow \tilde{v}_2$ and
$\tilde{e}_1 \leftrightarrow \tilde{e}_2$.
\end{proof}

\begin{corollary}
Let $G$ be a graph
and 
$$f_1, f_2 : G \to \mathbb{R}$$
be optimal discrete Morse functions.
Then,
\begin{enumerate}
\item $A_0^{f_1,f_2}(G)=\beta_0(G)$;
\item $A_1^{f_1,f_2}(G)\geq\beta_1(G)$.
\end{enumerate}

\end{corollary}

\begin{proof}
The first equation holds 
because there are no paths connecting vertices in different components.

\end{proof}

Note that in general cases, 
$A_q^{f_1,f_2}$ may be smaller than $\beta_q$.

\begin{proposition}
\label{0-dim case}
Let $G$ be a graph
and 
$$f_1, f_2 : G \to \mathbb{R}$$
be discrete Morse functions.
Let $\tilde{v}_1 $ and $\tilde{v}_2 $ be
$0$-dimensional $f_1$- and $f_2$-critical simplices,  respectively.
If $\tilde{v}_1$ is strongly connected to $\tilde{v}_2$,
then $\tilde{v}_1$ is not strongly connected to other 
critical $f_2$-critical simplices and $\tilde{v}_2$ is not strongly connected
to other $f_1$-critical simplices.
Moreover, the gradient path  connecting  $\tilde{v}_1$ to $\tilde{v}_2$ 
and the gradient path connecting $\tilde{v}_2$ to $\tilde{v}_1$ are unique.
\end{proposition}

\begin{proof}
Suppose $\tilde{v}_1 \leftrightarrow \tilde{v}_2$.
The second statement simply follows from the fact that
the outdegree of each vertex is at most $1$.
Hence, if there exists an $f_2$-critical $\tilde{v}_2'$ such that
$\tilde{v}_1 \to \tilde{v}_2'$,
then the gradient path $\gamma_{1,2'}:\tilde{v}_1 \stackrqarrow{V_2(0,1)}{} \tilde{v}_2' $
is either a sub-path of  
$\gamma_{1,2}:\tilde{v}_1 \stackrqarrow{V_2(0,1)}{} \tilde{v}_2 $,
or it contains $\gamma_{1,2'} $.
Both cases contradict with the definition of gradient path, because there is no
critical simplex on the gradient path.
\end{proof}

From the proof above we can see that
there may be two different $f_1$-critical vertices
$v_1$, $v_1'$ connecting to 
one $f_2$-critical vertex $v_2$,
but $v_2$ can not be connected to both of them.
In Proposition \ref{optimal}, we assumed the functions to be 
optimal.
Next, we remove the conditions on $f_1$ and $f_2$
and instead 
impose restrictions on the graph $G$
 by assuming that $G$ is a forest.

\begin{lemma}
Let $G$ be a  forest, 
and 
$$f_1, f_2 : G \to \mathbb{R}$$
be discrete Morse functions, $q=0,1$.
If $\tilde{\sigma}^{(q)}_1$ is strongly connected to
$\tilde{\sigma}^{(q)}_2$,
then they are not strongly connected to other 
critical simplices of $f_2$ and  $f_1$, respectively.
Moreover,  the gradient path  connecting  $\tilde{\sigma}^{(q)}_1$ 
and $\tilde{\sigma}^{(q)}_2$ 
to the gradient path connecting $\tilde{\sigma}^{(q)}_2$
to $\tilde{\sigma}^{(q)}_1$ are unique.
\end{lemma}

\begin{proof}
The case  $q=0$ follows from Proposition 4.5.
  When $q=1$,
assume that 
$\tilde{e}_1 \leftrightarrow \tilde{e}_2$ and
$\tilde{e}_1 \to \tilde{e}_2'$.
We aim to  prove  $\tilde{e}_2' \not \to \tilde{e}_1$.
Suppose, for contradiction, that $\tilde{e}_2' \to \tilde{e}_1$.

Since 
$\tilde{e}_1 \to \tilde{e}_2$ and 
$\tilde{e}_1 \to \tilde{e}_2'$,
we have two gradient paths 
$\gamma_{1,2}:\tilde{e}_1 \stackrqarrow{V_1(0,1)}{} \tilde{e}_2 $ and 
$\gamma_{1,2'}:\tilde{e}_1 \stackrqarrow{V_1(0,1)}{} \tilde{e}_2' $.
These paths either start from the same non-critical face of $\tilde{e}_1$ or 
from to two different non-critical faces  of $\tilde{e}_1$.
Since $G$ is a forest,
gradient paths $\gamma_{2,1}:\tilde{e}_2 \stackrqarrow{V_2(0,1)}{} \tilde{e}_1 $ and 
$\gamma_{2',1}:\tilde{e}_2' \stackrqarrow{V_2(0,1)}{} \tilde{e}_1$
have to be the reverse paths
of $\gamma_{1,2}$ and 
$\gamma_{1,2'}$,
respectively.

In the former case, 
either $\gamma_{1,2}$ or $\gamma_{1,2'}$ is a sub-path of the other. 
This contradicts the definition of a gradient path.
In the latter case, it
 contradicts the fact that $\tilde{e}_1$ 
cannot be part of both gradient pairs associated with its two faces.

The uniqueness of gradient paths is trivial.

\end{proof}

Note that in the above lemma,  although we 
suppose that  $G$ is a forest, 
we can relax the condition to 
``the subgraph
consisting of all paths connecting $\tilde{\sigma}_1 $ and $\tilde{\sigma}_2 $ does not contain cycles".

\begin{proposition}
\label{cyclic}
Let $G$ be a  cyclic graph, 
and 
$$f_1, f_2 : G \to \mathbb{R}$$
be discrete Morse functions.
Then, the following statements hold.

\begin{enumerate}
\item If $\sigma_1$ is connected to $\sigma_2$,
then the gradient path connecting them is unique.
\item $A_0^{f_1,f_2}(G)=A_1^{f_1,f_2}(G).$

\end{enumerate}

\end{proposition}

\begin{proof}
Statement (1) is trivial because there are always two paths connecting 
$\tilde{e}_1 $ and $\tilde{e}_2 $,
and one of them contains $\tilde{v}_1 $, 
where $(\tilde{v}_1 ,\infty)$ is a persistence pair of $f_1$,
hence it is not a gradient path.
Additionally,  Proposition 4.5 gives a proof to the case $q=0.$

To prove  statement (2),  suppose that  
$ (\tilde{v}_1,\infty)$ are $ (\tilde{e}_1,\infty)$ 
are persistence pairs for the cycle
in $f_1$.
Without loss of  generality, we 
order the $f_1$-critical simplices by considering the paths
connecting $ \tilde{e}_1$ and $ \tilde{v}_1$ as follows $(n,m\geq 1)$:
$$\gamma_1: \tilde{e}_1,v_1^1, e_1^1, \ldots,  v_{n}^1, e_{n}^1,\tilde{v}_1,$$
and
$$\gamma_2: \tilde{e}_1, v_1^{2}, e_1^2, \ldots,  v_{m}^2, e_{m}^2,\tilde{v}_1.$$
It is easy to check that the statement holds for the cases  $n=0$ or $m=0$.
Note that  $1$-dimensional critical simplices  always  appear after $0$-dimensional critical simplices
in the above sequence.

For any interval containing the endpoints starting from $v_i$ to $e_i$ in $\gamma_1$ or $\gamma_2$, 
 the $f_1$ flow is either trivial when $v_i \prec e_i$,  or is from $e_i$ to $v_i$ otherwise.
Similarly, the same  result holds when  considering the interval starting from $e_i$ to $v_{i+1}$,
 $\tilde{e}_1$ to $v_1^1$ and  $e_{n}^1$ to $\tilde{v}_1.$
We investigate the types of $f_2$-critical simplices and the variation of 
$A_q^{f_1,f_2}(G)$ in the interval.
Let  $\overline{v}$ and  $\overline{e}$ be $0$- and $1$-dimensional $f_2$-critical 
simplices, respectively.

\begin{enumerate}
\item When there is only one $f_2$-critical $1$-dimensional simplex $\overline{e}$ in the 
interval,
we have $e_i \leftrightarrow \overline{e}$, and there is no other strong connection in the 
interval.
 Hence, $A_1^{f_1,f_2}$ increases by one, and 
 $A_0^{f_1,f_2}$ does not change.
 
 \item When  there is only one $f_2$-critical $0$-dimensional simplex $\overline{v}$ in the 
interval,
we have $v_i \leftrightarrow \overline{v}$, and there is no other strong connection in the 
interval.
 Hence, $A_0^{f_1,f_2}$ increases by one, and 
 $A_1^{f_1,f_2}$ does not change.

 \item Suppose that there are multiple $f_2$-critical simplices with
  $\overline{e}$ being the first one and $\overline{v}$  being the last one in the interval.
 Then, $e_i \leftrightarrow \overline{e}$, and $v_i \leftrightarrow \overline{v}$.
 Hence, both $A_1^{f_1,f_2}$ and  $A_0^{f_1,f_2}$ increase by one.

 \item Suppose that there are multiple $f_2$-critical simplices with
  $\overline{v}$ being the first one and $\overline{e}$ 
   being the last one in the interval.
 Then,  there is no strong connection in the interval.
 Hence, neither $A_1^{f_1,f_2}$  nor  $A_0^{f_1,f_2}$ change.

 \item Suppose that there are multiple $f_2$-critical simplices with
  $\overline{e}$ being the first one and $\overline{e'}$
    being the last one in the interval.
 Then, $e_i \leftrightarrow \overline{e}$.
 Hence, $A_1^{f_1,f_2}$ increases by one, and 
 $A_0^{f_1,f_2}$ does not change.
 
 \item Suppose that there are multiple $f_2$-critical simplices with
  $\overline{v}$ being the first one and $\overline{v'}$  being the last one in the interval.
 Then, $v_i \leftrightarrow \overline{v}$.
 Hence, $A_0^{f_1,f_2}$ increases by one, and 
 $A_1^{f_1,f_2}$ does not change.

\end{enumerate}

Since $G$ is a cyclic graph, $C_0^{f_2}(G)=C_1^{f_2}(G).$
In cases (1), (5),  the number of  $1$-dimensional $f_2$-critical  
simplices contained in the interval is one more than $0$-dimensional $f_2$-critical  
simplices contained in the interval.
In cases (2), (6), the number of  $0$-dimensional $f_2$-critical  
simplices contained in the interval is one more than $1$-dimensional $f_2$-critical  
simplices contained in the interval.
In cases (3), (4), the number of  $0$-dimensional $f_2$-critical  
simplices contained in the interval is the same as $1$-dimensional $f_2$-critical  
simplices contained in the interval.

Hence, 
if we sum up all the intervals, then the 
sum of cases (1) and (5)  must equal  the sum of cases (2) and (6).
Therefore,  the difference between 
$A_0^{f_1,f_2}(G)$ and $A_1^{f_1,f_2}(G)$ equal the difference between 
$C_0^{f_2}(G)$ and $C_1^{f_2}(G)$, 
and
$A_0^{f_1,f_2}(G)=A_1^{f_1,f_2}(G).$
\end{proof}

Proposition 4.7 can be generalized by considering the set of gradient paths connecting 
$1$-dimensional critical simplices and $0$-dimensional critical simplices.
In \cite{Discrete Morse theory on graphs}, the \textit{rooted tree} of a 
discrete Morse function on $G$ is introduced to study the
structure of the gradient vector field. 
Moreover, in \cite{chari}, 
the rooted tree is applied to investigate the 
complex of discrete Morse functions. 
We cite the definition as follows.
For a detailed explanation 
and comprehensive discussion on the rooted tree and its applications, 
the reader may refer 
Section 4 of  \cite{Discrete Morse theory on graphs}.

\begin{definition}
Let $G$ be a  graph, 
and 
$f: G \to \mathbb{R}$
be discrete Morse functions.
Given any $0$-critical element $v$,  
the set of all gradient paths 
rooted in it 
forms a tree called the\textit{ tree rooted in v},  denoted as $T^f_v$.
When the context is clear,  we omit $f$.

Since any two of such rooted trees are disjoint,
we call the set consisting of all rooted trees $\{T_v\}$  the \textit{rooted forest}.
The rooted forest can be obtained by 
removing all critical edges of $f$.

\end{definition}

The tree rooted in $v$ contains  
two sets of gradient paths: the gradient paths
 that originate from the faces of $1$-dimensional critical simplices to $v$
and the gradient paths that do not start from any face of 
$1$-dimensional critical simplices.
The former set is denoted $\{\overline{\gamma_{e, v}}\}_e,$ where
$e$ is a $1$-dimensional critical simplex.
On the other hand,  the latter set is denoted $\overline{\gamma_{v}}$.
 $\overline{\gamma_{v}}$ consists of branches ending at 
$v$, with each branch represented $B_i$, $i=1,2\ldots n$. 
Note that the branch set may be empty  for some graphs,
for example, the cyclic graphs we consider above.
Hence, the branch set is also the main difficulty to generalize the proof
for cyclic graphs to general graphs.

We call the union $\overline{\gamma_{e, v}}\cup \overline{\gamma_{v}}$,
which is a subgraph of $T_v$,
the\textit{ tree of $e$ rooted in $v$},
denoted $T_{e, v}$.

\begin{example}
Figure \ref{f1} depicts 
an example of a discrete Morse function, denoted as $f_1$,
on  graph $G$.
The gradient vector field on $G$ 
is represented by 
arrows.
 Critical simplices 
 within the graph, 
 marked in red, 
 consist of three $0$-dimensional simplices $v_1^1,v_1^2,v_1^3$,
and three $1$-dimensional simplices $e_1^1,e_1^2,e_1^3$.

The rooted tree of vertex $v_1^2$ is illustrated in 
Figure \ref{rooted tree}.
The rooted tree comprises
 gradient paths $\overline{\gamma_{e_1^1, v_1^2}},$ 
 $\overline{\gamma_{e_1^2, v_1^2}}$,
  $\overline{\gamma_{e_1^3, v_1^2}}$
along with an additional branch. 
Note that  $\overline{\gamma_{e_1^3, v_1^2}}$
is contained in $\overline{\gamma_{e_1^1, v_1^2}}$ and  
 $\overline{\gamma_{e_1^2, v_1^2}}$.
\end{example}

\begin{figure}[ht]

\centering

\begin{tikzpicture}[>=stealth, thick, scale=1.25,
  decoration={markings, mark=at position 0.5 with {\arrow{>}}}
]
\foreach \i in {1,...,6} {
    \pgfmathsetmacro{\angle}{360/6*\i}
    \node[draw, circle, fill=black, inner sep=2pt] (A\i) at (\angle:2) {};
  }
  \foreach \i [remember=\i as \lasti (initially 6)] in {1,...,6} {
    \ifnum\i=4 
      \ifnum\lasti=3
        \draw[postaction={decorate}] (A\lasti) -- (A\i); 
      \fi
    \else
      \ifnum\i=5 
        \ifnum\lasti=4
          \draw[red] (A\lasti) -- (A\i) node[midway, above,black] {$e_1^2$} ; 
        \fi
      \else
        \ifnum\i=6
          \ifnum\lasti=5
            \draw[red] (A\i) -- (A\lasti) node[midway, above,black] {$e_1^3$}; 
          \fi
        \else
          \draw[postaction={decorate}] (A\lasti) -- (A\i); 
        \fi
      \fi
    \fi
  }

  \path let \p1 = (A2) in \pgfextra{\xdef\yTwo{\y1}};
  \path let \p1 = (A4) in \pgfextra{\xdef\yFour{\y1}};

\foreach \i in {7,8,9} {
    \pgfmathsetmacro{\xcoord}{-5.5 + (\i - 7) * 1.5} 
    \def\nodecolor{black}
    \ifnum\i=7
      \def\nodecolor{red} 
      \node at (\xcoord, \yTwo - 0.3cm) {$v_1^1$};
    \fi
    \node[draw, circle, fill=\nodecolor, inner sep=2pt] (A\i) at (\xcoord,\yTwo) {};
  }

 \draw[red] (A7) -- (A8) node[midway, above,black] {$e_1^1$}; 
  \draw[postaction={decorate}] (A8) -- (A9);
  \draw[postaction={decorate}] (A9) -- (A2); 

  \foreach \i in {10,11,12} {
    \pgfmathsetmacro{\xcoord}{-7 + (\i - 9) * 1.5} 
    \def\nodecolor{black}
    \ifnum\i=12
      \def\nodecolor{red} 
      \node at (\xcoord, \yFour - 0.3cm) {$v_1^2$};
    \fi
    \node[draw, circle, fill=\nodecolor, inner sep=2pt] (A\i) at (\xcoord,\yFour) {};
  }

  \draw[postaction={decorate}] (A10) -- (A11);
  \draw[postaction={decorate}] (A11) -- (A12);
  \draw[postaction={decorate}] (A4) -- (A12); 
  
  \foreach \i in {13,14,15} {
    \pgfmathsetmacro{\xcoord}{-3.5 + (\i - 9) * 1.5} 
    \def\nodecolor{black}
    \ifnum\i=15
      \def\nodecolor{red} 
      \node at (\xcoord, \yFour - 0.3cm) {$v_1^3$};
    \fi
    \node[draw, circle, fill=\nodecolor, inner sep=2pt] (A\i) at (\xcoord,\yFour) {};
  }

  \draw[postaction={decorate}] (A14) -- (A15); 
  \draw[postaction={decorate}] (A13) -- (A14);
  \draw[postaction={decorate}] (A5) -- (A13); 
\end{tikzpicture}
\caption{Graph $G$ with discrete Morse function $f_1$.}
\label{f1}
\end{figure}
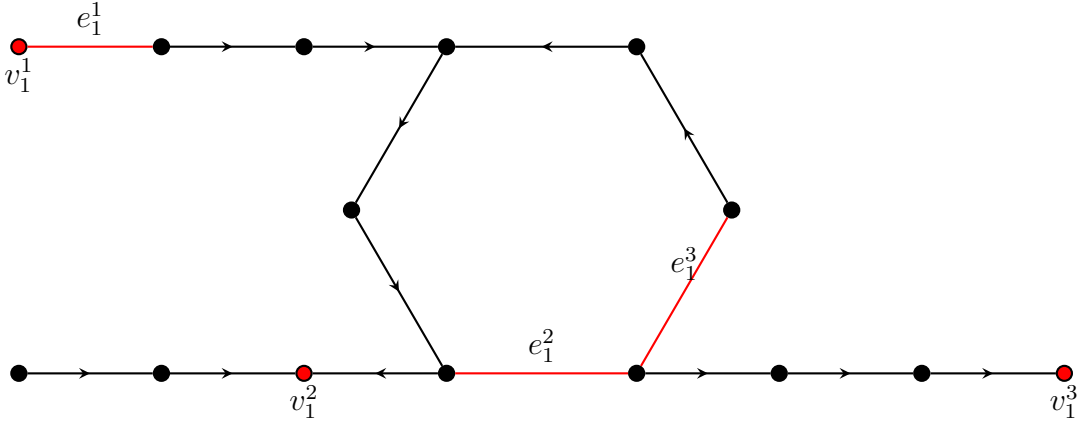

\begin{theorem}
\label{main theorem}
Let $G$ be a  graph, 
and 
$$f_1, f_2 : G \to \mathbb{R}$$
be discrete Morse functions.
Then, 
$$A_0^{f_1,f_2}(G)- A_1^{f_1,f_2}(G) =\beta_0(G)- \beta_1(G)=\chi(G),$$
where $\chi(G)$ is the  Euler characteristic of the graph $G$.

\end{theorem}

\begin{proof}
Let $v^1$ and $e^1$ be $0$-, $1$-dimensional $f_1$-critical simplices,
respectively.
Following the approach used in the proof of Proposition \ref{cyclic},
we consider the $f_2$-critical sequence of  simplices in 
the tree of $e^1$ rooted in $v^1$
$T^{f_1}_{e^1, v^1}=\overline{\gamma_{e^1, v^1}}\cup \overline{\gamma_{v^1}}$.
 If $e^1$ is also  $f_2$-critical, then  
 $e^1\leftrightarrow e^1$ via trivial paths.
Note that if there is any $f_2$-critical edge $e^2$ in $\overline{\gamma_{e^1, v^1}}$,
then $e^1\to e^2$.
Also,  if there is any $f_2$-critical edge $e^2$ in $\overline{\gamma_{v^1}}$,
then there is no $f_1$-critical edge that is connected to $e^2$.
Finally,  if there is any $f_2$-critical vertex $v^2$ in $T_{e^1, v^1}$,
then $v^2\to v^1$.

First, we discuss  the branch set $\overline{\gamma_{v^1}}$. 
 Let $ B_i$ be a branch starting at $v_i^1$ ending at $v^1$.
We consider the sequence 
of $f_2$-critical simplices between 
$v_i^1$ and $v^1$, always 
beginning with a $0$-dimensional simplex.
We introduce two additional cases in addition to 
the cases (1) to (6) from the proof of  Proposition 4.7.
\begin{enumerate}\setcounter{enumi}{6}
\item When the sequence of $f_2$-critical simplices starts at a 
$0$-dimensional simplex and ends at a 
$1$-dimensional simplex, 
neither $A_0$ nor $A_1$  changes.
\item When the sequence of $f_2$-critical simplices starts at a $0$-dimensional simplex 
ends at a 
$0$-dimensional simplex, 
$A_0$ increases by one while $A_1$ does not change.

\end{enumerate}
 
Next, we  modify the cases (1) to (6) from the proof of  Proposition 4.7 for
$\overline{\gamma_{e^1, v^1}}$.
In the context of  a general graph $G$,
an $f_2$-critical  simplex 
$e_2$ may  not be connected to $e^1$,
even if $e_2$ is the first 
  $f_2$-critical  edge of the $f_2$-critical
  sequence in $\overline{\gamma_{e^1, v^1}}$.
  Also, 
  the $f_2$-critical
  sequence may consists of consecutive 
   $f_2$-critical edges or even all $f_2$-critical edge, such as
  $e_2^1, e_2^2, \ldots, e_2^k$.
  
Suppose that $e_2$ is the first   $f_2$-critical  edge of the $f_2$-critical sequence
in $\overline{\gamma_{e^1, v^1}}$,
and $e_2$ is not connected to $e_1$.
The $f_2$-gradient flow then transitions from 
one or two faces of $e_2$ to one or two branches of 
$\overline{\gamma_{v^1}}$, respectively.
 This case is addressed by the branch case described above.

It is straightforward to  check that for each
 case in the tree $T^{f_1}_{e^1, v^1}$,
the variation in
$A_0^{f_1,f_2}(G)-A_1^{f_1,f_2}(G)$ equals the
 difference between the number of  $f_2$-critical 
 vertices and the number of  $f_2$-critical edges.
We can sum up the result by varying $e^1$ to obtain the result
for the rooted tree $T_{v_1}^{f_1}$.
Since any two of such rooted trees 
 are disjoint, 
 aggregating the individual results of
each rooted tree leads to the conclusion that
 $$A_0^{f_1,f_2}(G)- A_1^{f_1,f_2}(G) =\beta_0(G)- \beta_1(G)=\chi(G),$$
the  Euler characteristic of $G$.
\end{proof}


\begin{example}
We present an example to elucidate Theorem \ref{main theorem}.
Consider another discrete Morse function, denoted as $f_2$,
applied to the graph $G$ shown in Figure \ref{an example}.
In this instance, 
there are four $f_2$-critical
$0$-dimensional simplices $v_2^1,v_2^2,v_2^3,v_2^4$,
and 
four $f_2$-critical
 $1$-dimensional simplices $e_2^1,e_2^2,e_2^3,e_2^4$ of $f_2$.

Upon examination, 
one can identify the following strong connections:
$v_1^1 \leftrightarrow v_2^1$,
$v_1^2 \leftrightarrow v_2^2$,
$v_1^3 \leftrightarrow v_2^4$,
$e_1^1 \leftrightarrow e_2^1$,
$e_1^3 \leftrightarrow e_2^2$,
$e_1^3 \leftrightarrow e_2^4$.
This example
 notably demonstrates a scenario 
 where one critical simplex 
 is strongly 
 connected to two distinct critical simplices. 
Therefore, 
$A_0^{f_1,f_2}(G)- A_1^{f_1,f_2}(G)=3-3=\beta_0(G)- \beta_1(G)=1-1=0=\chi(G)$.
\end{example}

\begin{figure}[ht]

\centering

\begin{tikzpicture}[>=stealth, thick, scale=1.25,
  decoration={markings, mark=at position 0.5 with {\arrow{>}}}
]
\foreach \i in {1,...,6} {
    \pgfmathsetmacro{\angle}{360/6*\i}
    \node[draw, circle, fill=black, inner sep=2pt] (A\i) at (\angle:2) {};
     \ifnum\i=5 
      \node[fill=white] (A\i) at (\angle:2) {};
      \fi
  }
  \foreach \i [remember=\i as \lasti (initially 6)] in {1,...,6} {
    \ifnum\i=4 
      \ifnum\lasti=3
        \draw[postaction={decorate}] (A\lasti) -- (A\i); 
      \fi
    \else
      \ifnum\i=5 
        \ifnum\lasti=4
          \draw[white] (A\lasti) -- (A\i) node[midway, above,white] {$e_2^2$} ; 
        \fi
      \else
        \ifnum\i=6
          \ifnum\lasti=5
            \draw[white] (A\i) -- (A\lasti) node[midway, above,white] {$e_2^3$}; 
          \fi
        \else
          \draw[postaction={decorate}] (A\lasti) -- (A\i); 
        \fi
      \fi
    \fi
  }

  \path let \p1 = (A2) in \pgfextra{\xdef\yTwo{\y1}};
  \path let \p1 = (A4) in \pgfextra{\xdef\yFour{\y1}};

\foreach \i in {8,9} {
    \pgfmathsetmacro{\xcoord}{-5.5 + (\i - 7) * 1.5} 
    \def\nodecolor{black}
    \node[draw, circle, fill=\nodecolor, inner sep=2pt] (A\i) at (\xcoord,\yTwo) {};
  }

 \draw[white] (A7) -- (A8) node[midway, above,white] {$e_2^1$}; 
  \draw[postaction={decorate}] (A8) -- (A9);
  \draw[postaction={decorate}] (A9) -- (A2); 

  \foreach \i in {10,11,12} {
    \pgfmathsetmacro{\xcoord}{-7 + (\i - 9) * 1.5} 
    \def\nodecolor{black}
    \ifnum\i=12
      \def\nodecolor{red} 
      \node at (\xcoord, \yFour - 0.3cm) {$v_1^2$};
    \fi
    \node[draw, circle, fill=\nodecolor, inner sep=2pt] (A\i) at (\xcoord,\yFour) {};
  }

  \draw[postaction={decorate}] (A10) -- (A11);
  \draw[postaction={decorate}] (A11) -- (A12);
  \draw[postaction={decorate}] (A4) -- (A12); 

\end{tikzpicture}
\caption{The rooted tree $T_{v_1^2}^{f_1}$.}
\label{rooted tree}
\end{figure}
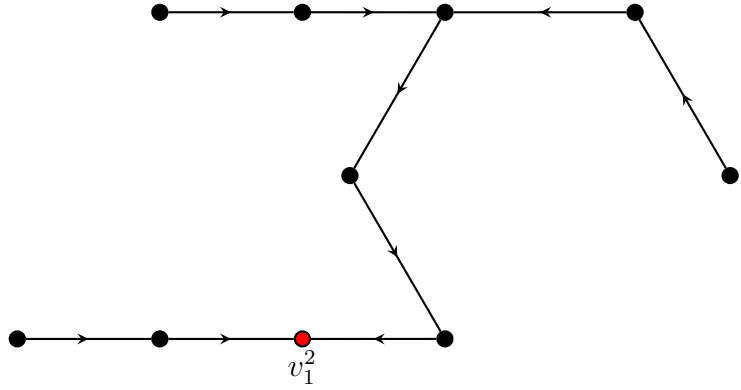

\begin{figure}[ht]

\centering

\begin{tikzpicture}[>=stealth, thick, scale=1.25,
  decoration={markings, mark=at position 0.5 with {\arrow{>}}}
]
\foreach \i in {1,...,6} {
    \pgfmathsetmacro{\angle}{360/6*\i}
    \def\nodecolor{black}
    \ifnum\i=1
      \def\nodecolor{red}
    \fi
    \node[draw, circle, fill=\nodecolor, inner sep=2pt] (A\i) at (\angle:2) {};
    \ifnum\i=1
      \node at (\angle:2.3) {$v_2^2$};
    \fi
  }
  \foreach \i [remember=\i as \lasti (initially 6)] in {1,...,6} {
    \ifnum\i=2 
      \ifnum\lasti=1 
        \draw[postaction={decorate}] (A\i) -- (A\lasti); 
      \fi
    \else
      \ifnum\i=3 
        \ifnum\lasti=2 
          \draw[postaction={decorate}] (A\i) -- (A\lasti); 
        \fi
      \else
        \ifnum\i=4 
          \ifnum\lasti=3
            \draw[red] (A\lasti) -- (A\i) node[midway, above,black] {$e_2^2$};  
          \else
            \draw[postaction={decorate}] (A\lasti) -- (A\i); 
          \fi
        \else
          \draw[postaction={decorate}] (A\lasti) -- (A\i); 
        \fi
      \fi
    \fi
  }

  \path let \p1 = (A2) in \pgfextra{\xdef\yTwo{\y1}};
  \path let \p1 = (A4) in \pgfextra{\xdef\yFour{\y1}};

\foreach \i in {7,8,9} {
    \pgfmathsetmacro{\xcoord}{-5.5 + (\i - 7) * 1.5} 
    \def\nodecolor{black}
    \ifnum\i=7
      \def\nodecolor{red} 
      \node at (\xcoord, \yTwo - 0.3cm) {$v_2^1$};
    \fi
    \node[draw, circle, fill=\nodecolor, inner sep=2pt] (A\i) at (\xcoord,\yTwo) {};
  }

 \draw[red] (A7) -- (A8) node[midway, above,black] {$e_2^1$}; 
  \draw[postaction={decorate}] (A8) -- (A9);
  \draw[postaction={decorate}] (A9) -- (A2); 

  \foreach \i in {10,11,12} {
    \pgfmathsetmacro{\xcoord}{-7 + (\i - 9) * 1.5} 
    \def\nodecolor{black}
    \ifnum\i=10
      \def\nodecolor{red} 
      \node at (\xcoord, \yFour - 0.3cm) {$v_2^3$};
    \fi
    \node[draw, circle, fill=\nodecolor, inner sep=2pt] (A\i) at (\xcoord,\yFour) {};
  }

  \draw[red] (A10) -- (A11) node[midway, above,black] {$e_2^3$};  
  \draw[postaction={decorate}] (A11) -- (A12);
  \draw[postaction={decorate}] (A12) -- (A4); 
  
  \foreach \i in {13,14,15} {
    \pgfmathsetmacro{\xcoord}{-3.5 + (\i - 9) * 1.5} 
    \def\nodecolor{black}
    \ifnum\i=15
      \def\nodecolor{red} 
      \node at (\xcoord, \yFour - 0.3cm) {$v_2^4$};
    \fi
    \node[draw, circle, fill=\nodecolor, inner sep=2pt] (A\i) at (\xcoord,\yFour) {};
  }

  \draw[red] (A15) -- (A14) node[midway, above,black] {$e_2^4$}; 
  \draw[postaction={decorate}] (A14) -- (A13);
  \draw[postaction={decorate}] (A13) -- (A5); 
\end{tikzpicture}
\caption{ Graph $G$ with discrete Morse function $f_2$.}
\label{an example}
\end{figure}
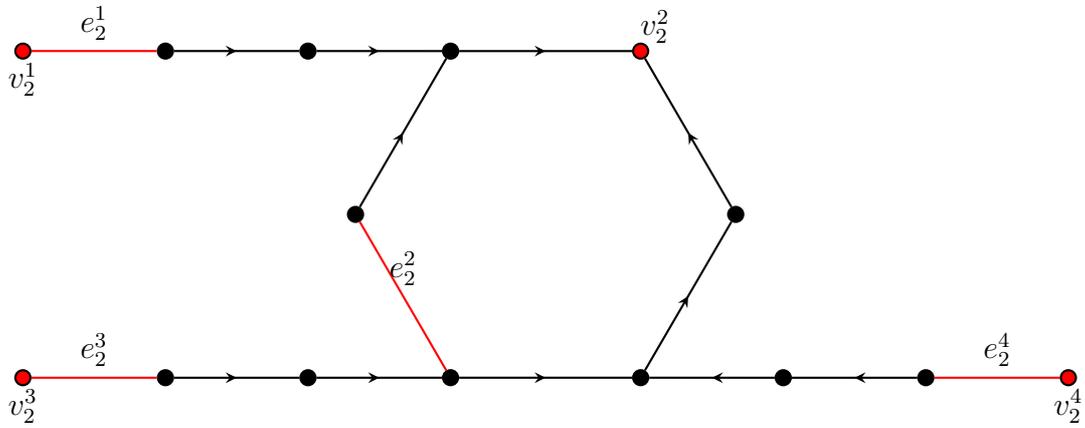

\newpage


\begin{acknowledgement*}

I express my 
gratitude to the 
editors and referees 
for their hard work and 
insightful feedback, 
which are pivotal in enhancing the
 quality of the manuscript. 
 Special thanks are 
 due to Prof. Toru Ohmoto (Waseda Univ.) 
 and Prof. Akira Koyama (Waseda Univ.) 
 for their exceptional guidance and 
 invaluable support throughout this series of work.
\end{acknowledgement*}


\end{document}